\documentclass{amsart}
\DeclareFontFamily{OML}{script}{}
\DeclareFontShape{OML}{script}{m}{it}
{ <5-20> rsfs10 }{}
\DeclareMathAlphabet{\mathscript}{OML}{script}{m}{it}

\renewcommand{\mathcal}[1]{{\mathscript #1}\hspace{0.2ex}}
\usepackage{color}
\ifx\red\undefined
\newcommand{\red}{\color{red}\huge}

\fi
\usepackage[ansinew]{inputenc}
\usepackage[encapsulated]{CJK}

\usepackage{graphicx,graphics}
\usepackage{cite}
\usepackage{pifont}
\usepackage{amsthm}

\usepackage{float}
\usepackage{amscd}
\usepackage{amsmath}
\usepackage{latexsym}
\usepackage{amsfonts}
\usepackage{amssymb}
\usepackage{color}
\usepackage{multicol}
\usepackage{bm}
\allowdisplaybreaks[4]
\usepackage{url}

\newcommand{\re}[1]{\mbox{\rm$($\ref{#1}$)$}}
\newcommand{\dis}{\displaystyle}
\newcommand{\m}{\hspace{1em}}
\newcommand{\mm}{\hspace{2em}}
\newcommand{\p}{\partial}

\newcommand{\Rmnum}[1]{\uppercase\expandafter{\romannumeral #1}}

\renewcommand{\epsilon}{\varepsilon}

\ifx\text\undefined
\newcommand{\text}{\mbox}
\fi
\ifx\operatorname\undefined
\newcommand{\operatorname}{\mathop}
\fi

\newcommand\be{\begin{equation}}
\newcommand\ee{\end{equation}}
\newcommand\bea{\begin{eqnarray}}
\newcommand\eea{\end{eqnarray}}
\newcommand\beaa{\begin{eqnarray*}}
\newcommand\eeaa{\end{eqnarray*}}
\newcommand{\dif}{\mathrm{d}}

\newenvironment{eqa}{\begin{equation}%
  \begin{array}{rcl}}{\end{array}\end{equation}}

\newcommand\beqa{\begin{eqa}}
\newcommand\eeqa{\end{eqa}}

\numberwithin{equation}{section}

\renewcommand{\tilde}{\widetilde}

\newtheorem{thm}{Theorem}[section]

\newtheorem{lem}[thm]{Lemma}

\newtheorem{rem}{Remark}[section]

\newcommand{\void}[1]{}

\numberwithin{equation}{section}

\begin{document}
\title{On the first bifurcation point for a free boundary problem modeling small arterial plaque\footnote{\today}
}
\author{Xinyue Evelyn Zhao}\author{Bei Hu}
\address{Department of Applied and Computational Mathematics and Statistics, 
University of Notre Dame, Notre Dame, IN 46556, USA}
\email{xzhao6@nd.edu, b1hu@nd.edu}
\maketitle

\begin{abstract}
Atherosclerosis occurs when plaque clogs the arteries. It is a leading cause of death in the United States and worldwide. In this paper, we study the bifurcation for a highly nonlinear and highly coupled PDE model of plaque formation in the early stage of atherosclerosis. The model involves LDL and HDL cholesterols, macrophage cells as well as foam cells, with the interface separating the plaque and blood flow region being a free boundary. We establish the first bifurcation point for the system corresponding to $n=1$ mode. The symmetry-breaking stationary solution studied in this paper might be helpful in understanding why there exists arterial plaque that is often accumulated more on one side of the artery than the other.
\end{abstract}

\section{Introduction}
Atherosclerosis is the clogging of artery from the build-up of plaque, which is originated from a small one. In the process it causes the hardening of the arteries and induces heart attacks.  
Every year about 735,000 Americans have a heart attack, and about 610,000 people die of heart diseases in the United States --- that is 1 in every 4 deaths (cf., \cite{web4,web3}).

Several mathematical models which describe  this process have been developed and studied recently; see \cite{CEMR, CMT,FHplaque1,FHHplaque, FHplaque2, mckay2005towards, mukherjee2019reaction,zhao3} and references therein. All of these models incorporate the critical role of the ``bad'' cholesterols, low density lipoprotein (LDL), and the ``good'' cholesterols, high density lipoprotein (HDL), in determining whether plaque will grow or shrink. The more
sophisticated model \cite{FHplaque1,FHplaque2} includes 17 variables, while a simplified
model \cite{FHHplaque,zhao3} combines some of these variables. In order to carry out theoretical analysis, the simplified model is used in this paper.

The process of plaque formation is as follows: 
when a lesion develops in the inner surface of the arterial wall, LDL and HDL are allowed to move into the intima and then oxidized by free radicals. Oxidized LDL would trigger endothelial cells to secrete chemoattractant proteins that attract macrophage cells (M) from the blood, and macrophage cells
would engulf oxidized LDL to become foam cells (F). As foam cells accumulates in the artery, plaque gradually builds up. The effect of oxidized LDL on plaque growth can be reduced by HDL in two ways: (a) HDL can remove harmful bad cholesterol out from the foam cells and revert foam cells back into macrophage cells; and (b) HDL also competes with LDL on free radicals, decreasing the amount of radicals that are available to oxidize LDL.

In the present model, we let
\begin{equation*}
\begin{split}
&L\text{ = concentration of LDL,}\hspace{5em} H\text{ = concentration of HDL,}\\
&M\text{ = density of macrophage cells,}\hspace{2em} F\text{ = density of foam cells.}
\end{split}
\end{equation*}
It is a very good approximation in assuming that the artery is a very long circular cylinder with radius being 1 (after normalization). We consider a circular cross section of the artery, which is 2-space dimensional. As can be seen in Fig. \ref{plaque}, the cross section is divided into two regions: blood flow region $\Sigma(t)$ and plaque region $\Omega(t)$, with a moving boundary $\Gamma(t)$ separating these two regions. The variables $L,H,M,F$ satisfy the following equations in the plaque region $\{\Omega(t), t>0\}$ (cf., \cite[Chapters 7 and 8]{mathbio} and \cite{FHHplaque}):
\begin{gather}
\frac{\partial L}{\partial t} - \Delta L = - k_1 \frac{ML}{K_1 + L} - \rho_1 L,\label{p1}\\
\frac{\partial H}{\partial t} - \Delta H =- k_2 \frac{HF}{K_2 + F} - \rho_2 H,\label{p2}\\
\frac{\partial M}{\partial t} - D\Delta M + \nabla\cdot(M \vec{v}) = -k_1\frac{ML}{K_1+L} + k_2\frac{HF}{K_2+F} + \lambda\frac{ML}{\gamma+H}-\rho_3 M,\label{p3}\\
\frac{\partial F}{\partial t} - D\Delta F + \nabla\cdot(F \vec{v}) = k_1\frac{ML}{K_1+L}-k_2\frac{HF}{K_2+F}-\rho_4 F.\label{p4}
\end{gather}
The above system \re{p1} --- \re{p4} includes the aforementioned transitions between macrophage cells ($M$) and foam cells ($F$), and their relationship with $H$ and $L$. The extra term $\lambda\frac{ML}{\gamma+H}$ in equation \re{p3} is phenomenological: the factor $ML$ accounts for the formation of foam cells, while the inhibition factor $1/(\gamma+H)$ describes  the fact that by oxidizing with free radicals, $H$ removes some of the radicals that are available to oxidize $L$.

We assume that the density of cells in the plaque is approximately a constant, and take
\begin{equation}
    M+F \equiv M_0\hspace{2em}\text{in } \Omega(t).\label{p5}
\end{equation}
Due to cell migration into and out of the plaque, the total number of cells keeps changing. Based on \re{p5}, cells are continuously ``pushing'' each other, which gives rise to an internal pressure $p$ among the cells. The internal pressure $p$ is associated with the velocity $\vec{v}$ in \re{p3} and \re{p4}. Since we treat the plaque as porous medium, we take the Darcy's law
\begin{equation}
    \vec{v} = -\nabla p \mm (\mbox{the proportional constant is normalized to 1}),\label{p6}
\end{equation}
Combining \re{p3} -- \re{p6}, we derive
\begin{equation}
    -\Delta p = \frac{1}{M_0}\Big[\lambda\frac{(M_0-F)L}{\gamma+H}-\rho_3(M_0-F) - \rho_4 F\Big].\label{peqn}
\end{equation}
Using the assumption \re{p5}, we can replace $M$ by $M_0 - F$ in \re{p1} -- \re{p4}, hence the model only consists of 4 PDEs, for $L$, $H$, $F$, and $p$, respectively. In particular, based on \re{peqn}, the equation for $F$ is
\begin{equation}
    \label{Feqn}
    \frac{\p F}{\p t}-D\Delta F - \nabla F\cdot \nabla p = k_1\frac{(M_0-F)L}{K_1+L}-k_2\frac{HF}{K_2+F}-\lambda\frac{F(M_0-F)L}{M_0(\gamma+H)}+(\rho_3-\rho_4)\frac{(M_0-F)F}{M_0}.
\end{equation}

In terms of boundary conditions, we assume no-flux condition on the blood vessel wall ($r=1$) for all variables (no exchange through the blood vessel): 
\begin{equation}
\frac{\partial L}{\partial r} = \frac{\partial H}{\partial r} = \frac{\partial F}{\partial r} =  \frac{\partial p}{\partial r} = 0 \hspace{2em}\text{at } r=1;\label{p7}
\end{equation}
while on the free boundary $\Gamma(t)$, we use the Robin boundary conditions:
\begin{eqnarray}
&\frac{\partial L}{\partial {\bf n}} + \beta_1 (L-L_0) = 0\hspace{2em} &\text{on }\Gamma(t),\label{p8}\\
&\frac{\partial H}{\partial {\bf n}}+\beta_1 (H-H_0) = 0\hspace{2em}&\text{on }\Gamma(t),\label{p9}\\
&\frac{\partial F}{\partial {\bf n}}+\beta_2 F = 0\hspace{2em}&\text{on }\Gamma(t),\label{p11}\\
&p = \kappa \hspace{2em} &\text{on }\Gamma(t),\label{p12}
\end{eqnarray}
where ${\bf n}$ denotes the outward unit normal vector for $\Gamma(t)$ which points inward to the blood region (as shown in Fig. \ref{plaque}), and $\kappa$ is the corresponding mean curvature in the direction of ${\bf n}$ (i.e., $\kappa=-\frac{1}{R(t)}$ if $\Gamma(t)=\{r=R(t)\}$). The cell-to-cell adhesiveness constant in front of $\kappa$ in \re{p12} is normalized to 1. The boundary conditions \re{p8} and \re{p9} are based upon the fact that the concentrations of $L$ and $H$ in the blood are $L_0$ and $H_0$, respectively; and the meaning of \re{p11} is similar: there are, of course, no foam cells in the blood.

Finally, we assume that the velocity is continuous up to the boundary, so that the free boundary $\Gamma(t)$ moves in the direction of ${\bf n}$ with velocity $\vec{v}$; on the foundation of \re{p6}, the normal velocity of the free boundary is defined by
\begin{equation}
    \label{p13}
    V_n = -\frac{\partial p}{\partial {\bf n}}\hspace{2em}\text{on }\Gamma(t).
\end{equation}

\vspace{10pt}

\begin{multicols}{2}
\begin{figure}[H]
\centering
\includegraphics[height=1.5in]{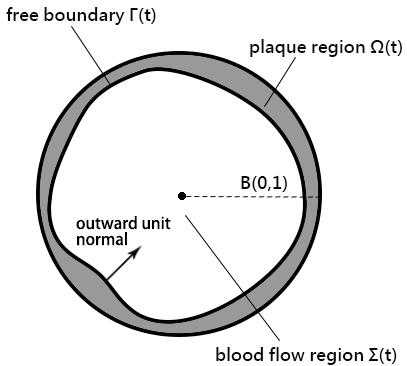}
\caption{\small The cross section of an artery.}
\label{plaque}
\end{figure}

\begin{figure}[H]
\centering
\includegraphics[height=1.5in]{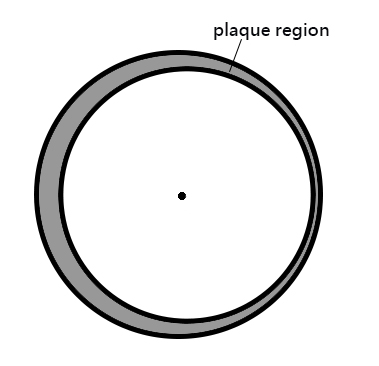}
\caption{\small A symmetry-breaking solution in the $n=1$ bifurcation branch}
\label{plaque-b}
\end{figure}
\end{multicols}

For the system \re{p1} -- \re{p13}, Friedman et al. \cite{FHHplaque} studied the radially symmetric case and justified the existence of a unique radially symmetric steady state solution in a small ring-region $1-\epsilon \le r\le 1$, which describes the arterial plaque in the early stage of atherosclerosis. Later on Zhao and Hu \cite{zhao3} carried out the bifurcation analysis for the system and established a branch of bifurcation points for $n\ge 2$. More specifically, they used $\mu= \frac1\epsilon[\lambda L_0-\rho_3(\gamma+H_0)]$ as the bifurcation parameter, and showed that for $n\ge2$, there exists a unique 
\begin{equation}
    \label{bifurpt}
    \mu_n = (\gamma + H_0) n^2 (1-n^2) + O(n^5 \epsilon),
\end{equation}
if $\mu_n > \mu_c$ ($\mu_c$ is defined in equation (2.9) in \cite{zhao3}), then $\mu = \mu_n$ is a bifurcation point of the symmetry-breaking stationary solution when $\epsilon$ is small enough. 

Based on further exploration and refinements of the estimates in \cite{zhao3}, we shall prove in this paper that $\mu_1 = O(\epsilon)$ is also a bifurcation point for the system \re{p1} -- \re{p13} by verifying the Crandall-Rabinowitz Theorem (Theorem \ref{CRthm}). Our main result is stated in the following theorem.

\begin{thm}\label{main}
Assume that $\mu_c<0$ and $\beta_1\neq \beta_2$. For $\mu_1 = O(\epsilon)$ defined as the solution to the equation \re{mun} with $n=1$, we can find a small $E>0$, such that for $0<\epsilon<E$, 
then $\mu=\mu_1$ is a bifurcation point of the symmetry-breaking stationary solution for the system \re{p1} -- \re{p13}. Moreover, the free boundary of this bifurcation solution is of the form
\begin{equation}
    r = 1-\epsilon + \tau \cos(\theta) + o(\tau),\mm \text{where } |\tau| \ll \epsilon.
\end{equation}
\end{thm}

In recent years, considerable research works have been done on bifurcation analysis based upon the Crandall-Rabinowitz Theorem (see \cite{CE3, FFBessel, FH3, FR2, Hao1, Hao2, angio1, Fengjie, Hongjing, Zejia, W, WZ2,zhao2,song}). In these works, the solutions on the $n=1$ bifurcation branch are just $\epsilon$-translations of the origin of the radially symmetric solution; after transformation of coordinates, this kind of solutions are still radially symmetric. Therefore, the $n=1$ case is always ignored in previous works on the bifurcation analysis. For our problem, however, the circumstances are {\em very different.} There are two boundaries for the system \re{p1} -- \re{p13}, with the outer boundary $r=1$ being fixed. Since the outer boundary is fixed, all the perturbations make changes only on the inner free boundary. Due to this special geometry, the solutions on the $n=1$ bifurcation branch are non-radially symmetric, as shown in the Figure~\ref{plaque-b}.

We want to emphasise that, for technical reasons, Theorem \ref{main} was not established in \cite{zhao3}. The primary reason is that, by \re{bifurpt}, both $\mu_0$ and $\mu_1$ are of order $O(\epsilon)$; thus \re{bifurpt} is not enough to guarantee  $\mu_0 \neq \mu_1$, and hence the assumption (2) of the Crandall-Rabinowitz Theorem cannot be verified. The main goal of this paper is to derive more sophisticated estimates for $\mu_0$ and $\mu_1$, with which we can then verify the Crandall-Rabinowitz Theorem.

From mathematical point of view, $\mu = \mu_1 = O(\epsilon)$ is the first bifurcation point, which often coincides with the change of stability for the system. In fact, it has been proved in \cite{FHHplaque} that the radially symmetric plaque would disappear if $\mu < 0$, and remain persistent if $\mu > 0$; hence it is likely that the stability of the radially symmetric solution would change around $\mu = 0$. More importantly, $\mu = \mu_1 = O(\epsilon)$ is also the most significant bifurcation point biologically, as the arterial plaque is often accumulated more on one side of the artery in reality (see Figures in \cite{plaque-fig1,plaque-fig2,plaque-fig3}).

The structure of this paper is as follows. We collect some preliminaries in Section 2, and give some useful estimates  in Section 3. The proof to the main result, Theorem \ref{main}, is presented in Section 4.

\section{Preliminaries}
\subsection{A small radially symmetric stationary solution}
We denote the radially symmetric stationary solution of the system \re{p1} -- \re{p13} by $(L_*, H_*, F_*, p_*)$. Dropping all the time derivatives, and writing the system in polar coordinates in the domain $\Omega_*=\{1-\epsilon<r<1\}$, we obtain the equation for $(L_*,H_*,F_*,p_*)$:
\begin{eqnarray}
&- \Delta L_* = - k_1 \frac{(M_0-F_*)L_*}{K_1 + L_*} - \rho_1 L_* &\text{in }\Omega_*,\label{r1}\\
&- \Delta H_* =- k_2 \frac{H_*F_*}{K_2 + F_*} - \rho_2 H_* &\text{in }\Omega_*,\label{r2}\\
&\label{r3}
    -D\Delta F_* - \frac{\p F_*}{\p r}\frac{\p p_*}{\p r}= k_1\frac{(M_0-F_*)L_*}{K_1+L_*}-k_2\frac{H_*F_*}{K_2+F_*}-\lambda\frac{F_*(M_0-F_*)L_*}{M_0(\gamma+H_*)}+(\rho_3-\rho_4)\frac{(M_0-F_*)F_*}{M_0} &\text{in }\Omega_*,\\
&-\Delta p_* = \frac{1}{M_0}\Big[\lambda\frac{(M_0-F_*)L_*}{\gamma+H_*}-\rho_3(M_0-F_*) - \rho_4 F_*\Big]\hspace{1em} &\text{in }\Omega_*,\label{r4}\\
&\frac{\partial L_*}{\partial r} = \frac{\partial H_*}{\partial r} = \frac{\partial F_*}{\partial r} =  \frac{\partial p_*}{\partial r} = 0, \hspace{2em}& r=1,\label{r5}\\
&\m -\frac{\partial L_*}{\partial r} + \beta_1 (L_*-L_0) = 0, \hspace{1em}  -\frac{\partial H_*}{\partial r} + \beta_1 (H_*-H_0)=0, \hspace{1em} -\frac{\partial F_*}{\partial r} + \beta_2 F_* = 0,\hspace{1em}&r=1-\epsilon,\label{r6}\\
&p_* = -\frac{1}{1-\epsilon},\hspace{1em}  &r=1-\epsilon,\label{r7}\\
&  \frac{\partial p_*}{\partial r}=0, &r=1-\epsilon.\label{r7a}
\end{eqnarray}

There are many parameters in the system. As in \cite{zhao3}, we keep all parameters fixed except $L_0$ and $\rho_4$.
For convenience, rather than using $L_0$, we use  $\mu = \frac{1}{\epsilon}[\lambda L_0 - \rho_3(\gamma+H_0)]$ as our bifurcation parameter and let $\rho_4 = \rho_4(\mu)$. 
The existence theorem for the radially symmetric solution has been established in \cite{zhao3}, and is stated as follows:
\begin{thm}
Define
$\mu_c = \frac{ \rho_3}{\beta_1}\Big\{  
 (\gamma + H_0) \Big( \frac{\lambda k_1 M_0  }{\lambda K_1+\rho_3(\gamma+H_0)}+ {\rho_1 } \Big) -\rho_2 H_0 \Big\}$,
for every $\mu_* > \mu_c$ and $\mu_c<\mu<\mu_*$, we can find a small $\epsilon^* > 0$, and for each $0<\epsilon < \epsilon^*$, there exists a unique $\rho_4$ such that the system \re{r1} -- \re{r7a} admits a unique solution $(L_*, H_*, F_*, p_*)$.
\end{thm}
\begin{rem}
By ODE theories, the solution $(L_*, H_*, F_*, p_*)$ can be extended to the bigger region $\Omega_{2\epsilon}=\{1-2\epsilon <r <1\}$ while maintaining $C^\infty$ regularity. For notational convenience, we still use $(L_*, H_*, F_*, p_*)$ to denote the extended solution. 
\end{rem}

\subsection{Bifurcation theorem}
Next we state a useful theorem which is  critical in studying bifurcations.
\begin{thm}\label{CRthm}  {\bf (Crandall-Rabinowitz theorem, \cite{crandall})}
Let $X$, $Y$ be real Banach spaces and $\mathcal{F}(\cdot,\cdot)$ a $C^p$ map, $p\ge 3$, of a neighborhood $(0,\mu_0)$ in $X \times \mathbb{R}$ into $Y$. Suppose
\begin{itemize}
\item[(1)] $\mathcal{F}(0,\mu) = 0$ for all $\mu$ in a neighborhood of $\mu_0$,
\item[(2)] $\mathrm{Ker} \,\mathcal{F}_x(0,\mu_0)$ is one dimensional space, spanned by $x_0$,
\item[(3)] $\mathrm{Im} \,\mathcal{F}_x(0,\mu_0)=Y_1$ has codimension 1,
\item[(4)] $\mathcal{F}_{\mu x}(0,\mu_0) x_0 \notin Y_1$.
\end{itemize}
Then $(0,\mu_0)$ is a bifurcation point of the equation $\mathcal{F}(x,\mu)=0$ in the following sense: In a neighborhood of $(0,\mu_0)$ the set of solutions $\mathcal{F}(x,\mu) =0$ consists of two $C^{p-2}$ smooth curves $\Gamma_1$ and $\Gamma_2$ which intersect only at the point $(0,\mu_0)$; $\Gamma_1$ is the curve $(0,\mu)$ and $\Gamma_2$ can be parameterized as follows:
$$\Gamma_2: (x(\epsilon),\mu(\epsilon)), |\epsilon| \text{ small, } (x(0),\mu(0))=(0,\mu_0),\; x'(0)=x_0.$$
\end{thm}

\subsection{Preparations for the bifurcation theorem}
In order to tackle the existence of symmetry-breaking stationary solutions to system \re{p1} -- \re{p13}, we'd like to apply the Crandall-Rabinowitz theorem. The preparations are the same as in \cite{zhao3}, and are similar as those in \cite{CE3, FFBessel, FH3, FR2, Hao1, Hao2, angio1, Fengjie, Hongjing, Zejia, W, WZ2,zhao2,song}.

We consider a family of perturbed domains $\Omega_\tau=\{1-\epsilon+ \tilde{R} < r < 1\}$ and denote the corresponding inner boundary to be $\Gamma_\tau$, where $\tilde{R}=\tau S(\theta)$, $|\tau| \ll \epsilon$ and $|S|\le 1$. Let $(L,H,F,p)$ be the solution of the system:
\begin{eqnarray} 
&- \Delta L = - k_1 \frac{(M_0-F)L}{K_1 + L} - \rho_1 L\hspace{1em} &\text{in } \Omega_\tau,\label{b1}\\
&- \Delta H =- k_2 \frac{HF}{K_2 + F} - \rho_2 H\hspace{1em} &\text{in } \Omega_\tau,\label{b2}\\
&-D\Delta F - \nabla F\cdot \nabla p = k_1\frac{(M_0-F)L}{K_1+L}-k_2\frac{HF}{K_2+F}-\lambda\frac{F(M_0-F)L}{M_0(\gamma+H)}+(\rho_3-\rho_4)\frac{(M_0-F)F}{M_0}\hspace{1em} &\text{in }\Omega_\tau,\label{b3}\\
&-\Delta p = \frac{1}{M_0}\Big[\lambda\frac{(M_0-F)L}{\gamma+H}-\rho_3(M_0-F) - \rho_4 F\Big]\hspace{1em} &\text{in }\Omega_\tau,\label{b4}\\
&\frac{\partial L}{\partial r} = \frac{\partial H}{\partial r} = \frac{\partial F}{\partial r} =  \frac{\partial p}{\partial r} = 0, \hspace{2em}& r=1,\label{b5}\\
&\frac{\partial L}{\partial {\bf n}} + \beta_1 (L-L_0) = 0, \hspace{1em}  \frac{\partial H}{\partial {\bf n}} + {\beta_1}(H-H_0)=0, \hspace{1em} \frac{\partial F}{\partial {\bf n}} + \beta_2 F = 0&\text{on } \Gamma_\tau,\label{b6}\\
&p = \kappa &\text{on } \Gamma_\tau.\label{b7}
\end{eqnarray}
The existence and uniqueness of such a solution is guaranteed in \cite{zhao3}. We then define function $\mathcal{F}$ by
\begin{equation}
    \label{F}
    \mathcal{F}(\tau S,\mu) = -\frac{\p p}{\p {\bf n}}\Big|_{\Gamma_\tau},
\end{equation}
We know that $(L,H,F,p)$ is a symmetry-breaking stationary solution if and only if $\mathcal{F}(\tau S,\mu)=0$. Next we introduce the Banach spaces:
\begin{gather}
    X^{l+\alpha} = \{S\in C^{l+\alpha}(\Sigma), S \text{ is $2\pi$-periodic in $\theta$}\},\nonumber\\
    \label{Banach}
    X^{l+\alpha}_1 = \text{closure of the linear space spanned by $\{\cos(n\theta),n=0,1,2,\cdots\}$ in $X^{l+\alpha}$}.
\end{gather}
It has been shown in \cite{zhao3} that the mapping $\mathcal{F}(\cdot,\mu): X^{l+4+\alpha}_1 \rightarrow X^{l+1+\alpha}_1$ is bounded for any $l > 0 $.

In order to apply the Crandall-Rabinowitz theorem, we need to compute the Fr\'echet derivatives of $\mathcal{F}$. For a fixed small $\epsilon$, we write the expansion of $(L,H,F,p)$ of order $\tau$ as follows:
\begin{gather}
    L = L_*+\tau L_1 + O(\tau^2),\label{expand1}\\
    H = H_*+\tau H_1 + O(\tau^2), \label{expand2}\\
    F = F_*+\tau F_1 + O(\tau^2), \label{expand3}\\
    p = p_*+\tau p_1 + O(\tau^2). \label{expand4}
\end{gather}
The rigorous justification for \re{expand1} -- \re{expand4} can be found in \cite{zhao3}. Substituting \re{expand1} -- \re{expand4} into \re{b1} -- \re{b7}, and dropping all the higher order terms in $\tau$, we  obtain the system for $(L_1,H_1,F_1,p_1)$, which is also called the linearized system of \re{b1} -- \re{b7}.

Set the perturbation $S(\theta) = \cos(n\theta)$, we are seeking solutions of the form
\begin{eqnarray}
    &L_1=L_1^n \cos(n\theta), \hspace{2em} &H_1=H_1^n\cos(n\theta),\label{nt1}\\
    &F_1=F_1^n \cos(n\theta), \hspace{2em} &p_1=p_1^n\cos(n\theta).\label{nt2}
\end{eqnarray}
From \cite[(4.7)-(4.16)]{zhao3}, we know that the equations for $(L_1^n, H_1^n, F_1^n, p_1^n)$ are (Recall that $\Omega_*$ denotes the annulus $1-\epsilon\le r\le 1$):
\begin{eqnarray}
    &-\frac{\p^2 L_1^n}{\p r^2}-\frac{1}{r}\frac{\p L_1^n}{\p r} + \frac{n^2}{r^2}L_1^n= f_5(L_1^n,H_1^n,F_1^n) \hspace{2em} &\text{in }\Omega_*,\label{L1n1}\\
    &-\frac{\p^2 H_1^n}{\p r^2}-\frac{1}{r}\frac{\p H_1^n}{\p r} + \frac{n^2}{r^2}H_1^n= f_6(L_1^n,H_1^n,F_1^n) \hspace{2em} &\text{in }\Omega_*,\label{H1n1}\\
    &-\frac{\p^2 F_1^n}{\p r^2}-\frac{1}{r}\frac{\p F_1^n}{\p r} + \frac{ n^2}{r^2}F_1^n = \frac1D\Big(f_7(L_1^n,H_1^n, F_1^n)+\frac{\p F_*}{\p r}\frac{\p p_1^n}{\p r} + \frac{\p F_1^n}{\p r} \frac{\p p_*}{\p r}\Big)
   \hspace{2em} &\text{in }\Omega_*,\label{F1n1}\\
   &-\frac{\p^2 p_1^n}{\p r^2}-\frac{1}{r}\frac{\p p_1^n}{\p r} + \frac{n^2}{r^2}p_1^n =  f_8(L_1^n,H_1^n,F_1^n) \mm &\text{in }\Omega_* ,\label{p1n1}\\
    &\frac{\p L_1^n}{\p r}=\frac{\p H_1^n}{\p r}=\frac{\p F_1^n}{\p r}=\frac{\p p_1^n}{\p r} = 0 & r=1,\label{bdy1n1}\\
    &-\frac{\p L_1^n}{\p r}+\beta_1 L_1^n=\Big(\frac{\p^2 L_*}{\p r^2}-\beta_1\frac{\p L_*}{\p r}\Big)\Big|_{r=1-\epsilon} &r=1-\epsilon,\label{bdyL1n1}\\
    &-\frac{\p H_1^n}{\p r}+\beta_1 H_1^n=\Big(\frac{\p^2 H_*}{\p r^2}-\beta_1\frac{\p H_*}{\p r}\Big)\Big|_{r=1-\epsilon} &r=1-\epsilon,\label{bdyH1n1}\\
    &-\frac{\p F_1^n}{\p r}+\beta_2 F_1^n=\Big(\frac{\p^2 F_*}{\p r^2}-\beta_2\frac{\p F_*}{\p r}\Big)\Big|_{r=1-\epsilon} &r=1-\epsilon,\label{bdyF1n1}\\
    &p_1^n = \frac{1-n^2}{(1-\epsilon)^2} &r=1-\epsilon,\label{bdyp1n1}
\end{eqnarray}
where $f_5,f_6,f_7,$ and $f_8$ can all be bounded by linear functions of $|L_1^n|$, $|H_1^n|$, and $|F_1^n|$. In particular, $f_8$ is expressed as
\begin{equation}
    \label{f8}
    f_8(L_1^n,H_1^n,F_1^n) = \frac{1}{M_0}\Big[\lambda\frac{(M_0-F_*)L^n_1}{\gamma+H_*} -\lambda\frac{L_* F^n_1}{\gamma+H_*} -\lambda\frac{(M_0-F_*)L_*H^n_1}{(\gamma+H_*)^2} + (\rho_3-\rho_4)F_1^n\Big],
\end{equation}

Since $\mathcal{F}(0,\mu) = \frac{\p p_*}{\p r}\big|_{r=1-\epsilon} = 0$ by \re{r7a}, we can utilize the expansions \re{expand1} -- \re{expand4} to derive (more rigorous proof can be found in \cite{zhao3})
\begin{equation*}
    \begin{split}
        \mathcal{F}(\tau S,\mu)- \mathcal{F}(0,\mu)= -\frac{\p p}{\p {\bm n}}\Big|_{\Gamma_\tau}&=  \frac{\p (p_*+\tau p_1)}{\p r}\Big|_{r=1-\epsilon+\tau S} + O(|\tau|^2 \|S\|_{C^{4+\alpha}(\Sigma)})\\
        &=\tau\Big[\frac{\p^2 p_*}{\p r^2}\Big|_{r=1-\epsilon}S(\theta) + \frac{\p p_1}{\p r}\Big|_{r=1-\epsilon}\Big] + O(|\tau|^2 \|S\|_{C^{4+\alpha}(\Sigma)}),
    \end{split}
\end{equation*}
which leads to the Fr\'echet derivative of $\mathcal{F}$ at $(0,\mu)$ as below
\begin{equation*}
    \Big[\mathcal{F}_{\tilde{R}}(0,\mu)\Big]S(\theta) = \frac{\p^2 p_*}{\p r^2}\Big|_{r=1-\epsilon}S(\theta) + \frac{\p p_1}{\p r}\Big|_{r=1-\epsilon}.
\end{equation*}
Substituting in $S(\theta) = \cos(n\theta)$ and \re{nt2}, we further obtain
\begin{equation}\label{FreD}
      [\mathcal{F}_{\tilde{R}}(0,\mu)]\cos(n\theta) = \Big(\frac{\p^2 p_*(1-\epsilon)}{\p r^2} + \frac{\p p_1^n(1-\epsilon)}{\p r}\Big)\cos(n\theta).
\end{equation}
Therefore we denote   $\mu=\mu_n$ to be the solution to the equation
\begin{equation}
    \label{mun}
    \frac{\p^2 p_*(1-\epsilon)}{\p r^2} + \frac{\p p_1^n(1-\epsilon)}{\p r} = 0.
\end{equation}
Clearly, for each $n\ge 0$, $[\mathcal{F}_{\tilde{R}}(0,\mu)]\cos(n\theta) = 0$ if and only if $\mu = \mu_n$. It is shown that for $n\ge 2$,
$\mu=\mu_n$ is a bifurcation point.

\section{Useful Estimates and Lemmas}
A lot of estimates were derived in \cite{zhao3}. 
In order to derive a better estimates than \re{bifurpt} for $n=0$ and $1$, here we collect some of them (i.e., (2.11)--(2.13) and (4.47)--(4.52)) which will be useful in 
this paper. (Note that we only consider the cases when $n=0$ and 1, hence we don't need the order of $n$ in higher order terms of $\epsilon$.)

\begin{eqnarray}
 \label{Ls} L_*(r) & = & L_0 - \frac{\epsilon}{\beta_1}
 \Big( \frac{k_1 M_0 L_0}{K_1+L_0}+\rho_1 L_0\Big) + O(\epsilon^2) \\
 & = & \frac{\rho_3(\gamma+H_0)}\lambda + \epsilon \Big[ \frac\mu\lambda  \nonumber
  - \frac{\rho_3(\gamma+H_0)}{\beta_1}
 \Big( \frac{k_1 M_0  }{\lambda K_1+\rho_3(\gamma+H_0)}+\frac{\rho_1 }\lambda  \Big)\Big]+  O(\epsilon^2) \\
 & \triangleq & \frac{\rho_3(\gamma+H_0)}\lambda +\epsilon  L_*^1 + O(\epsilon^2),\nonumber  \\
 \label{Hs} H_*(r) & = & H_0  -   \epsilon  \frac{\rho_2 H_0 }{\beta_1} +  O(\epsilon^2)
\; \triangleq \; H_0 + \epsilon H_*^1 + O(\epsilon^2),\\
 \label{Fs} F_*(r) & = &  \frac\epsilon{\beta_2}  \frac{k_1 M_0 L_0}{D(K_1+L_0)}   +  O(\epsilon^2) \\ & = & \epsilon \; \frac{\rho_3(\gamma+H_0)}{\beta_2 D} \;
 \frac{k_1 M_0  }{\lambda K_1+\rho_3(\gamma+H_0)}+ O(\epsilon^2) 
\; \triangleq \; \epsilon F_*^1 + O(\epsilon^2).\nonumber\\
\label{L1n} L_1^n(r) &=& \frac{1}{\beta_1}\Big(\frac{\p^2 L_*}{\p r^2} - \beta_1 \frac{\p L_*}{\p r}\Big)\Big|_{r=1-\epsilon} + O(\epsilon) \; =\; \frac{\mu}{\lambda} - L_*^1 + O(\epsilon),\\
\label{H1n} H_1^n(r) &=& \frac{1}{\beta_1}\Big(\frac{\p^2 H_*}{\p r^2} - \beta_1 \frac{\p H_*}{\p r}\Big)\Big|_{r=1-\epsilon} + O(\epsilon) \; =\; - H_*^1 + O(\epsilon),\\
\label{F1n} F_1^n(r) &=& \frac{1}{\beta_2}\Big(\frac{\p^2 F_*}{\p r^2} - \beta_2 \frac{\p F_*}{\p r}\Big)\Big|_{r=1-\epsilon} + O(\epsilon) \; =\;  - F_*^1 + O(\epsilon).
\end{eqnarray}

The following lemma is from \cite{zhao3} (\textit{Lemma 2.2}). It includes a relationship among the parameters, which will be used later.
\begin{lem}
\label{rho-relation}
For the radially symmetric stationary solution $(L_*,H_*,F_*,p_*)$, the following estimates holds
\begin{equation}
    \label{rho41}
    \frac{M_0(\lambda L_*^1 - \rho_3 H_*^1) }{\gamma+ H_0} - \rho_4 F_*^1 = O(\epsilon),
\end{equation}
in other words, $\rho_4$ can be expressed as
\begin{equation}
    \label{rho42}
    \rho_4 = \frac{1}{F_*^1}\,\frac{M_0}{\gamma+H_0}(\lambda L_*^1 - \rho_3 H_*^1) + O(\epsilon).
\end{equation}
\end{lem}

Notice that equations \re{L1n1} -- \re{p1n1} are of similar structure, hence we denote the operator $\mathscript L_n \triangleq \frac{\p^2}{\p r^2} + \frac{1}{r}\frac{\p }{\p r} + \frac{n^2}{r^2}$. For this special operator, one can easily verify the following lemmas (special case of $n=0$ and $1$ in \cite[Lemma 4.2]{zhao3}, with the case $n=1$ modified to satisfy $\psi_1(1)=\psi_1'(1)=0$)
\begin{lem}\label{eqn-n}
  The general solution of ($\eta$ is a constant)
\begin{equation}
\label{psieqn}
\begin{split}
 & {\mathscript L}_n [\psi]\triangleq- \psi''- \frac1r \psi' + \frac{n^2}{r^2} \psi  = \eta + f(r) \mm 1-\epsilon < r < 1 , \\
 & \psi'(1) = 0, 
 \end{split}
\end{equation}
is given by 
\bea\label{sol}
 && \psi - \psi_1  =  \left\{ \begin{array}{ll} \dis
 A r + \frac1r \Big(A+K[f]'(1)\Big) + K[f](r)  \mm & n=1, \\
  A + K[f](r) & n= 0,
\end{array}\right. 
\eea
where
\bea \label{Kf} 
&&  K[f](r)  =  
 \left\{ \begin{array}{ll}
 \dis \frac{r}{2} \int_r^1  f(s) \, \dif s+
  \frac{r^{-1}}{2} \int_{1-\epsilon}^r  s^2 f(s) \, \dif s \hspace{1em}& n=1,\\
 \dis\rule{0pt}{18pt}- \int_{r}^1 \Big(\log\frac{s}r\Big)\; s f(s) \, \dif s & n= 0;
\end{array}\right.
 \eea
in addition, $\psi_1(r)$ satisfies
\begin{equation}\label{psi1eqn}
\begin{split}
 & {\mathscript L}_n[\psi_1] = - \psi_1''- \frac1r \psi_1' + \frac{n^2}{r^2} \psi_1  = \eta  \mm 1-\epsilon < r < 1 ,\\
 & \psi_1'(1) = 0, \mm \psi_1(1) = 0,
 \end{split}
\end{equation}
and is given by 
\bea \label{psi1} 
\psi_1 =  \left\{ \begin{array}{ll}
 \dis \eta\Big(-\frac1{6r}+\frac12 r - \frac13 r^2\Big) \mm& n=1, \\
 \vspace{-8pt}\\
 \dis\eta\Big( \frac{1-r^2}4 + \frac12\log r\Big) \hspace{1em}& n = 0.
 \end{array}\right.  
\eea 
The special solution $K[f]$ satisfies
\be 
 \label{estK} | K[f](r) | \le  \frac{\epsilon}{2} \|f\|_{L^\infty}, \mm
  | K[f]'(r) | \le  \frac{\epsilon}{2} \|f\|_{L^\infty} , \mm n=1,
\ee
and
\be 
 | K[f](r) | \le  \epsilon \|f\|_{L^\infty}, \mm
  | K[f]'(r) | \le  \epsilon \|f\|_{L^\infty} , \mm n = 0.
\ee
\end{lem}

The proof is the same as in \cite{zhao3}. We add one more restriction on $\psi_1$, i.e., $\psi_1(1) = 0$, but it doesn't affect the proof. Since $\psi_1(1) = \psi_1'(1)= 0$, and by \re{psi1eqn},
\[
 \psi_1''(1) = -\eta + \Big(-\frac1r \psi_1' + \frac{n^2}{r^2}\psi_1 \Big)\Big|_{r=1} = -\eta.
\]
In addition, differentiating the equation \re{psi1eqn} and evaluating at $r=1$, we further obtain
\[ 
 \psi_1'''(1) = \Big( -\frac{r \psi_1'' -\psi_1'}{r^2} + n^2 \frac{r^2 \psi_1' - 2r\psi_1}{r^4}\Big)\Big|_{r=1} = - \psi_1''(1)  =\eta.
\]
\void{
\begin{equation*}
    \psi_1''(1) = \left\{
    \begin{split}
        &\eta\Big(-\frac13 \frac1{r^3} - \frac23\Big)\Big|_{r=1} = -\eta\hspace{2em} &n=1,\\
        &\eta\Big(-\frac12 - \frac1{2r^2}\Big)\Big|_{r=1} = -\eta \hspace{2em}&n=0,
    \end{split}
    \right.
\end{equation*}
\begin{equation*}
    \psi_1'''(1) = \left\{
    \begin{split}
        &\frac\eta{r^4}\Big|_{r=1} = \eta\hspace{2em} &n=1,\\
        &\frac\eta{r^3}\Big|_{r=1} = \eta \hspace{2em}&n=0,
    \end{split}
    \right.
\end{equation*}
}
Based on the above two equations, it then follows from the Taylor series that, for $1-\epsilon\le r\le1$,
\begin{gather}
    \label{1} \psi_1(r) = \psi_1(1) + \psi_1'(1)(r-1) + \frac{\psi_1''(1)}{2}(r-1)^2 + \cdots = -\frac\eta2 (r-1)^2 + O(\epsilon^3),\\
    \label{2} \psi_1'(r) = \psi_1'(1) + \psi_1''(1)(r-1) + \frac{\psi_1'''(1)}{2}(r-1)^2 + \cdots = -\eta(r-1) + \frac\eta2(r-1)^2 +  O(\epsilon^3).
\end{gather}
In particular, we have
\begin{gather}
    \label{3}\psi_1(1-\epsilon) = -\frac{\eta}{2}\epsilon^2 + O(\epsilon^3),\\
    \label{4}\psi_1'(1-\epsilon) = \epsilon \eta + \frac\eta2 \epsilon^2 + O(\epsilon^3).
\end{gather}

\begin{lem}\label{eqn-n1} 
If in addition to \re{psieqn}, we further assume the boundary condition 
\begin{equation}
    \label{bdypsi1}
    -\psi'(1-\epsilon) + \beta \psi(1-\epsilon) = G,
\end{equation}
then the coefficient $A$ in \re{sol} can be explicitly computed as: for $n=1$
\be \label{A1}
A = \frac{G  + \psi_1'(1-\epsilon)  - \beta \psi_1(1-\epsilon)- \beta K[f](1-\epsilon) + K[f]'(1-\epsilon) - \frac{1}{(1-\epsilon)^2}K[f]'(1)-\frac{\beta}{1-\epsilon}K[f]'(1)}{-1+\frac{1}{(1-\epsilon)^2}+\beta(1-\epsilon) + \frac{\beta}{1-\epsilon}},
\ee
and for $n=0$,
\be\label{A0}
A = \frac{1}{\beta}\Big[G  + \psi_1'(1-\epsilon)  - \beta \psi_1(1-\epsilon)- \beta K[f](1-\epsilon) + K[f]'(1-\epsilon)\Big].
\ee
\end{lem}

\begin{lem}\label{eqn-n2}
If $f(r)=O(\epsilon)$ in \re{psieqn}, and the assumptions in Lemma \ref{eqn-n1} hold, then for $1-\epsilon \le r \le 1$,
\begin{eqnarray}
    \label{psisol1}  &&\psi(r) = \frac{G}{\beta} + \epsilon\Big(\frac{\eta}{\beta}-\frac{G}{\beta^2}\Big) + O(\epsilon^2) \mm n=1,\\
    \label{psisol2}  &&\psi(r) = \frac{G}{\beta} + \epsilon\frac{\eta}{\beta} + O(\epsilon^2) \hspace{5.6em} n=0.
\end{eqnarray}

\end{lem}
\begin{proof}
Based on Lemma \ref{eqn-n}, if $f(r)=O(\epsilon)$ in \re{psieqn}, we have in either $n=0$ or $n=1$ case,
\begin{equation*}
    K[f](r) = O(\epsilon^2),\mm K[f]'(r) = O(\epsilon^2).
\end{equation*}
Substituting into \re{A1} and \re{A0}, recalling also \re{3} and \re{4}, we obtain
\begin{equation*}
\begin{split}
    &A = \frac{G + \epsilon\eta + O(\epsilon^2) }{2(\beta + \epsilon) + O(\epsilon^2)} = \frac{G}{2\beta} + \epsilon\Big(\frac{\eta}{2\beta}-\frac{G}{2\beta^2}\Big) + O(\epsilon^2) \mm &n = 1,\\
    &A = \frac{G+\epsilon \eta + O(\epsilon^2)}{\beta} =\frac{G}{\beta} + \epsilon \frac{\eta}{\beta} + O(\epsilon^2)\mm & n=0.
    \end{split}
\end{equation*}
Since $A$ is the only coefficient in \re{sol}, we can now substitute the above expressions for $A$ into \re{sol} to derive, for $1-\epsilon \le r \le 1$,
\begin{equation*}
\begin{split}
    &\psi(r) = A\Big(r+\frac1r\Big) + O(\epsilon^2) = 2A + O(\epsilon^2)  = \frac{G}{\beta} + \epsilon\Big(\frac{\eta}{\beta}-\frac{G}{\beta^2}\Big) + O(\epsilon^2)\mm &n=1,\\
    &\psi(r) = A + O(\epsilon^2) = \frac{G}{\beta} + \epsilon \frac{\eta}{\beta} + O(\epsilon^2) \mm &n=0.
    \end{split}
\end{equation*}
This completes the proof.
\end{proof}
Notice that in this lemma the difference between $n=1$ and $n=0$ cases starts from $O(\epsilon)$ terms; furthermore, the difference in $O(\epsilon)$ terms is $\epsilon \frac{G}{\beta^2}$, which is determined only by $G$ and $\beta$.

Lemma \ref{eqn-n}, together with Lemmas \ref{eqn-n1} and \ref{eqn-n2}, are applied to equations for $L_1^n$, $H_1^n$ and $F_1^n$. Notice that the boundary condition for $p_1^n$ is 
$$p_1^n(1-\epsilon) = \frac{1-n^2}{(1-\epsilon)^2},$$
which is of a different form from \re{bdypsi1}, we hence need the following lemma that can be easily verified: 

\begin{lem}\label{eqn-n3}
If in addition to \re{psieqn}, we further assume the boundary condition for $\psi$
\begin{equation}
    \label{bdypsi2}
    \psi(1-\epsilon) = \frac{1-n^2}{(1-\epsilon)^2},
\end{equation}
then the coefficient $A$ in \re{sol} is solved as
\begin{equation*}
    \begin{split}
        &A = \frac{-\psi_1(1-\epsilon) - K[f](1-\epsilon) -\frac{1}{1-\epsilon}K[f]'(1)}{1-\epsilon + \frac{1}{1-\epsilon}}  \mm &n=1,\\
        &A = 
        - \psi_1(1-\epsilon) - K[f](1-\epsilon) \mm &n=0.
    \end{split}
\end{equation*}
\end{lem}

\begin{lem}
\label{eqn-n4}
If $f(r) = O(\epsilon^2)$ in \re{psieqn}, and the assumptions in Lemma \ref{eqn-n3} hold, then for $n=0,1$ and $1-\epsilon \le r \le 1$,
\begin{eqnarray}
    \psi'(1-\epsilon) = \epsilon \eta + \epsilon^2 \frac{\eta}{2} + O(\epsilon^3).
\end{eqnarray}
\end{lem}

\begin{proof}
If $f = O(\epsilon^2)$ in \re{psieqn}, by Lemma \ref{eqn-n} we have
\begin{equation*}
    K[f](r) = O(\epsilon^3),\mm K[f]'(r) = O(\epsilon^3).
\end{equation*}
In order to estimate $\psi'(1-\epsilon)$, we differentiate \re{sol} and evaluate the derivative at $r=1-\epsilon$,
\begin{equation*}
    \begin{split}
        &\psi'(1-\epsilon) = \psi_1'(1-\epsilon) + A - \frac{1}{(1-\epsilon)^2}\Big(A + K[f]'(1)\Big) + K[f]'(1-\epsilon) \mm &n=1,\\
        &\psi'(1-\epsilon) = \psi_1'(1-\epsilon) + K[f]'(1-\epsilon) \mm & n=0.
    \end{split}
\end{equation*}
Combining with the expression of $A$ in Lemma \ref{eqn-n3}, recalling also \re{1} -- \re{4}, we further derive
$A=O(\epsilon^2)$, and in both cases,
\begin{equation*}
    \begin{split}
        &\psi'(1-\epsilon) = \epsilon \eta + \epsilon^2\frac{\eta}{2} + O(\epsilon^3) \mm & n=0,1,
    \end{split}
\end{equation*}
which completes the proof.
\end{proof}

\section{Proof of Theorem 1.1}
As mentioned before, the key is to show $\mu_0\neq \mu_1$.
Since they are the same on the order of $O(1)$, we shall derive higher order approximations for $\mu_0$ and $\mu_1$. Hence the estimates \re{L1n} -- \re{F1n} for $L_1^n$, $H_1^n$ and $F_1^n$ are not enough for the purpose of this paper -- we need more information about the higher order terms in $\epsilon$. To do that, we denote
\begin{eqnarray}
    \label{L11n} L_1^n &=& \frac{\mu}{\lambda} - L_*^1 + \epsilon L_{11}^n + O(\epsilon^2),\\
    \label{H11n} H_1^n &=& -H_*^1 + \epsilon H_{11}^n + O(\epsilon^2),\\
    \label{F11n} F_1^n &=& -F_*^1 + \epsilon F_{11}^n + O(\epsilon^2),
\end{eqnarray}
and we proceed to estimate $L_{11}^n$, $H_{11}^n$ and $F_{11}^n$ based on Lemmas \ref{eqn-n}, \ref{eqn-n1} and \ref{eqn-n2}.

Recall first the equations for $L_1^n$, $H_1^n$ and $F_1^n$ are ($\Omega_*$ denotes the annulus $1-\epsilon\le r\le 1$):
\begin{eqnarray}
    &-\frac{\p^2 L_1^n}{\p r^2}-\frac{1}{r}\frac{\p L_1^n}{\p r} + \frac{n^2}{r^2}L_1^n= f_5(L_1^n,H_1^n,F_1^n) \hspace{2em} &\text{in }\Omega_*,\label{L1n11}\\
    &-\frac{\p^2 H_1^n}{\p r^2}-\frac{1}{r}\frac{\p H_1^n}{\p r} + \frac{n^2}{r^2}H_1^n= f_6(L_1^n,H_1^n,F_1^n) \hspace{2em} &\text{in }\Omega_*,\label{H1n11}\\
    &-\frac{\p^2 F_1^n}{\p r^2}-\frac{1}{r}\frac{\p F_1^n}{\p r} + \frac{ n^2}{r^2}F_1^n = \frac1D\Big(f_7(L_1^n,H_1^n, F_1^n)+\frac{\p F_*}{\p r}\frac{\p p_1^n}{\p r} + \frac{\p F_1^n}{\p r} \frac{\p p_*}{\p r}\Big)
   \hspace{2em} &\text{in }\Omega_*,\label{F1n11}\\
    &\frac{\p L_1^n}{\p r}=\frac{\p H_1^n}{\p r}=\frac{\p F_1^n}{\p r}=0 & r=1,\label{bdy1n11}\\
    &-\frac{\p L_1^n}{\p r}+\beta_1 L_1^n=\Big(\frac{\p^2 L_*}{\p r^2}-\beta_1\frac{\p L_*}{\p r}\Big)\Big|_{r=1-\epsilon} &r=1-\epsilon,\label{bdyL1n11}\\
    &-\frac{\p H_1^n}{\p r}+\beta_1 H_1^n=\Big(\frac{\p^2 H_*}{\p r^2}-\beta_1\frac{\p H_*}{\p r}\Big)\Big|_{r=1-\epsilon} &r=1-\epsilon,\label{bdyH1n11}\\
    &-\frac{\p F_1^n}{\p r}+\beta_2 F_1^n=\Big(\frac{\p^2 F_*}{\p r^2}-\beta_2\frac{\p F_*}{\p r}\Big)\Big|_{r=1-\epsilon} &r=1-\epsilon.\label{bdyF1n11}
\end{eqnarray}
For the right-hand sides of \re{L1n11} -- \re{F1n11}, we can write them as the form $\eta + O(\epsilon)$, and we shall claim that $\eta$ is independent of $n$. In fact, we notice that the $O(1)$ terms of $L_1^n(r)$, $H_1^n(r)$ and $F_1^n(r)$ in \re{L1n} -- \re{F1n} are constants, and are independent of $n$. Moreover, it has been proved in \cite{zhao3} that $\frac{\p F_*}{\p r}, \frac{\p p_*}{\p r} = O(\epsilon)$, and $\frac{\p p_1^n}{\p r}, \frac{\p F_1^n}{\p r}$ are both bounded. Hence the extra two terms in \re{F1n11}, $\frac{\p F_*}{\p r}\frac{\p p_1^n}{\p r}$ and $\frac{\p F_1^n}{\p r}\frac{\p p_*}{\p r}$, do not affect the $O(1)$ term $\eta$.

Using Lemma \ref{eqn-n2}, we find that the difference between $L_1^1(r)$ and $L_r^0(r)$ starts from $O(\epsilon)$ terms. As a matter of fact, by \re{psisol1}, \re{psisol2}, and \re{L1n}, 
\begin{equation*}
    L_1^1 - L_1^0 = -\frac{\epsilon}{\beta_1^2} \Big(\frac{\p^2 L_*}{\p r^2}-\beta_1\frac{\p L_*}{\p r}\Big)\Big|_{r=1-\epsilon} + O(\epsilon^2) = \frac{\epsilon}{\beta_1}\Big( L_*^1 - \frac{\mu}{\lambda}\Big) + O(\epsilon^2).
\end{equation*}
From \re{L11n}, $L_1^1 = \frac{\mu}{\lambda} - L_*^1 + \epsilon L_{11}^1 + O(\epsilon^2)$, and $L_1^0 = \frac{\mu}{\lambda} - L_*^1 + \epsilon L_{11}^0 + O(\epsilon^2)$; combining with the above equation, we further have
\begin{equation}
    \label{DL1n} 
    L_{11}^1 - L_{11}^0 = \frac{1}{\beta_1}\Big(L_*^1 - \frac{\mu}{\lambda}\Big).
\end{equation}
Similarly, we can derive
\begin{gather}
    \label{DH11n} H_{11}^1 - H_{11}^0 = \frac{1}{\beta_1} H_*^1,\\
    \label{DF11n} F_{11}^1 - F_{11}^0 = \frac{1}{\beta_2} F_*^1.
\end{gather}

\bigskip

Next we consider the equation for $p_1^n$:
\begin{eqnarray}
     &&-\frac{\p^2 p_1^n}{\p r^2}-\frac{1}{r}\frac{\p p_1^n}{\p r} + \frac{n^2}{r^2}p_1^n =  f_8(L_1^n,H_1^n,F_1^n) \mm \text{in }\Omega_* ,\label{p1n11}\\
     &&\frac{\p p_1^n}{\p r}(1) = 0,\mm  p_1^n(1-\epsilon) = \frac{1-n^2}{(1-\epsilon)^2},\label{bdyp1n11}
\end{eqnarray}
where
\begin{equation}
    \label{f81}
    f_8(L_1^n, H_1^n, F_1^n) = \frac{1}{M_0}\Big[\lambda\frac{(M_0-F_*)L^n_1}{\gamma+H_*} -\lambda\frac{L_* F^n_1}{\gamma+H_*} -\lambda\frac{(M_0-F_*)L_*H^n_1}{(\gamma+H_*)^2} + (\rho_3-\rho_4)F_1^n\Big].
\end{equation}
In order to apply Lemmas \ref{eqn-n}, \ref{eqn-n3} and \ref{eqn-n4}, we shall rewrite $f_8$ in the form $f_8 = \eta + f(r)$, where $f(r) = O(\epsilon^2)$. More specifically, we denote,
\begin{equation*}
    f_8(L_1^n, H_1^n, F_1^n) = \eta_n + O(\epsilon^2),
\end{equation*}
and we now proceed a long and tedious journal to compute $\eta_n = \eta_n(\epsilon)$. Substituting \re{Ls} -- \re{Fs} and \re{L11n} -- \re{F11n} all into \re{f81}, we have
\begin{equation*}
    \begin{split}
         M_0 f_8 \;=&\; \lambda\frac{(M_0-\epsilon F_*^1)(\frac{\mu}{\lambda}-L_*^1 + \epsilon L_{11}^n)}{\gamma+H_0 + \epsilon H_*^1} - \lambda \frac{(\frac{\rho_3(\gamma+H_0)}{\lambda}+\epsilon L_*^1)(-F_*^1 + \epsilon F_{11}^n)}{\gamma+H_0+\epsilon H_*^1} + (\rho_3-\rho_4)(-F_*^1 + \epsilon F_{11}^n)\\
         &- \lambda \frac{(M_0-\epsilon F_*^1)(\frac{\rho_3(\gamma+H_0)}{\lambda}+\epsilon L_*^1)(-H_*^1 + \epsilon H_{11}^n)}{(\gamma+H_0+\epsilon H_*^1)^2} + O(\epsilon^2)\\
         \;=&\;\Big(\frac{M_0(\mu-\lambda L_*^1)}{\gamma+H_0} + \frac{(-\epsilon F_*^1)(\mu - \lambda L_*^1)}{\gamma+H_0} + \frac{\lambda M_0 \epsilon L_{11}^n}{\gamma +H_0} - \frac{M_0(\mu-\lambda L_*^1)\epsilon H_*^1}{(\gamma+H_0)^2} \Big) - \Big(-\rho_3 F_*^1 - \frac{\lambda \epsilon L_*^1 F_*^1}{\gamma+H_0} \\
         &+\rho_3\epsilon F_{11}^n + \frac{\rho_3 F_*^1\epsilon H_*^1}{\gamma+H_0}\Big) + \Big(-(\rho_3-\rho_4)F_*^1 + (\rho_3-\rho_4)\epsilon F_{11}^n\Big) - \Big(-\frac{M_0\rho_3 H_*^1}{\gamma+H_0} + \frac{\epsilon F_*^1 \rho_3  H_*^1}{\gamma+H_0}\\
         &-\frac{\lambda M_0 \epsilon L_*^1 H_*^1}{(\gamma+H_0)^2} + \frac{M_0\rho_3 \epsilon H_{11}^n}{\gamma+H_0} + \frac{2M_0 \rho_3 H_*^1 \epsilon H_*^1}{(\gamma+H_0)^2}\Big) + O(\epsilon^2)\\
         \;=&\; \Big(\frac{M_0 \mu}{\gamma+H_0} - \frac{M_0(\lambda L_*^1 - \rho_3 H_*^1)}{\gamma+H_0} + \rho_4 F_*^1 \Big) + \epsilon \Big(-\frac{F_*^1 \mu}{\gamma + H_0} - \frac{M_0 H_*^1 \mu}{(\gamma+H_0)^2 } + 2\frac{\lambda L_*^1 F_*^1}{\gamma+ H_0} + 2\frac{\lambda M_0 L_*^1 H_*^1}{(\gamma+ H_0)^2}\\
         &-2\frac{\rho_3 F_*^1 H_*^1}{\gamma+H_0} - 2\frac{M_0\rho_3 (H_*^1)^2}{(\gamma+H_0)^2}\Big) + \epsilon\Big(\frac{\lambda M_0 L_{11}^n}{\gamma+H_0} -\frac{M_0\rho_3 H_{11}^n}{\gamma+H_0} - \rho_4 F_{11}^n \Big) + O(\epsilon^2),
    \end{split}
\end{equation*}
hence
\begin{equation}
    \label{etan}\begin{split}
        \eta_n =&\; \frac{1}{M_0}\Big[\Big(\frac{M_0 \mu}{\gamma+H_0} - \frac{M_0(\lambda L_*^1 - \rho_3 H_*^1)}{\gamma+H_0} + \rho_4 F_*^1 \Big) + \epsilon \Big(-\frac{F_*^1 \mu}{\gamma + H_0} - \frac{M_0 H_*^1 \mu}{(\gamma+H_0)^2 } + 2\frac{\lambda L_*^1 F_*^1}{\gamma+ H_0} \\
         &+ 2\frac{\lambda M_0 L_*^1 H_*^1}{(\gamma+ H_0)^2} -2\frac{\rho_3 F_*^1 H_*^1}{\gamma+H_0} - 2\frac{M_0\rho_3 (H_*^1)^2}{(\gamma+H_0)^2}\Big)\Big] + \frac{\epsilon}{M_0}\Big(\frac{\lambda M_0 L_{11}^n}{\gamma+H_0} -\frac{M_0\rho_3 H_{11}^n}{\gamma+H_0} - \rho_4 F_{11}^n \Big).
    \end{split}
\end{equation}
We notice that all the terms in the first bracket of \re{etan} are independent of $n$. Applying Lemma \ref{eqn-n4}, we immediately have
\begin{equation}
    \label{p1r}
    \frac{\p p_1^n(1-\epsilon)}{\p r} = \epsilon \eta_n + \epsilon^2 \frac{\eta_n}{2} + O(\epsilon^3).
\end{equation}

Now we are ready to prove our main result Theorem \ref{main}. 

\begin{proof}[\textbf{Proof of Theorem \ref{main}}]
What we need to do is to verity the four assumptions of the Crandall-Rabinowitz Theorem (Theorem \ref{CRthm}) at the point $\mu=\mu_1$. We choose the Banach spaces $X=X^{4+\alpha}_1$ and $Y = X^{1+\alpha}_1$. 
The differentiabilty of the map follows the same argument as in the papers \cite{zhao3, angio1, Hongjing,zhao2,song}. To begin with, the assumption (1) is naturally satisfied, and the assumption (4) can be justified in the same way as in \cite{zhao3}. For the remaining assumptions (2) and (3), we need to prove
\begin{equation}\label{obj}
    \mu_n \neq \mu_1 \mm \text{for }\forall n\neq 1. 
\end{equation}
From \cite{zhao3}, it has been established that there exists a bound $E_1>0$, when $0<\epsilon < E_1$, we have $\mu_n \neq \mu_1$ for each $n \ge 2$. Hence it suffices to show
\begin{equation}
    \label{1n0}
    \mu_0 \neq \mu_1,
\end{equation}
and we shall prove it by contradiction. 

Recall that $\mu_n$ is the solution to the equation 
\begin{equation*}
    \frac{\p^2 p_*(1-\epsilon)}{\p r^2} + \frac{\p p_1^n(1-\epsilon)}{\p r} = 0.
\end{equation*}
For the contrary, assuming $\mu_0 = \mu_1$, we then have 
\begin{equation}
\label{contrad}
    \frac{\p p_1^1(1-\epsilon)}{\p r} - \frac{\p p_1^0(1-\epsilon)}{\p r} = 0.
\end{equation}
On the other hand, it follows from \re{p1r} and \re{etan} that
\begin{equation*}
\begin{split}
    \frac{\p p_1^1(1-\epsilon)}{\p r} - \frac{\p p_1^0(1-\epsilon)}{\p r}  \;&=\; \epsilon (\eta_0 - \eta_1) + \frac{\epsilon^2}{2}(\eta_0 - \eta_1) + O(\epsilon^3)\\
    &=\; \frac{\epsilon^2}{M_0}\Big[\frac{\lambda M_0}{\gamma+H_0}(L_{11}^1 -L_{11}^0) -\frac{\rho_3 M_0}{\gamma+H_0}(H_{11}^1 - H_{11}^0) - \rho_4(F_{11}^1 - F_{11}^0) \Big] + O(\epsilon^3).
    \end{split}
\end{equation*}
Substituting into the estimates for $L_{11}^1 - L_{11}^0$, $H_{11}^1 - H_{11}^0$, and $F_{11}^1 - F_{11}^0$ in \re{DL1n} -- \re{DF11n}, recalling also \re{rho41} in Lemma \ref{rho-relation} and the fact that $\mu_0,\mu_1 = O(\epsilon)$ by \re{bifurpt}, we further have
\begin{equation*}
    \begin{split}
        \frac{\p p_1^1(1-\epsilon)}{\p r} - \frac{\p p_1^0(1-\epsilon)}{\p r}  \;&=\; \epsilon^2\Big[\frac{\lambda}{\gamma+H_0}\frac{1}{\beta_1}(L_*^1 - \frac{\mu_0}{\lambda}) - \frac{\rho_3}{\gamma+H_0}\frac{H_*^1}{\beta_1} - \frac{\rho_4}{M_0}\frac{F_*^1}{\beta_2}\Big] + O(\epsilon^3)\\
        &=\; \epsilon^2 \Big[\frac{1}{\beta_1}\frac{\lambda L_*^1 - \rho_3 H_*^1}{\gamma+H_0} - \frac{1}{\beta_2}\frac{\rho_4 F_*^1}{M_0} - \frac{\mu_0}{\beta_1(\gamma+H_0)}\Big] + O(\epsilon^3)\\
        &=\; \epsilon^2 \Big(\frac{1}{\beta_1} - \frac{1}{\beta_2}\Big)\frac{\rho_4 F_*^1}{M_0} + O(\epsilon^3).
    \end{split}
\end{equation*}
We have assumed that $\beta_1\neq\beta_2$. Since the sign of $\frac{\p p_1^1(1-\epsilon)}{\p r} - \frac{\p p_1^0(1-\epsilon)}{\p r}$ is dominated by $\Big(\frac{1}{\beta_1} - \frac{1}{\beta_2}\Big)\frac{\rho_4 F_*^1}{M_0}$, we can easily find $E_2 > 0$, such that when $0<\epsilon < E_2$, 
$$\frac{\p p_1^1(1-\epsilon)}{\p r} - \frac{\p p_1^0(1-\epsilon)}{\p r} \neq 0,$$
which contradicts with the statement \re{contrad}. Hence we have $\mu_1 \neq \mu_0$ when $\epsilon < E_2$.

By taking $E = \min\{E_1, E_2\}$, we finish showing \re{obj}. With \re{obj}, we now have
\beaa 
 && \text{Ker} \,\mathcal{F}_{\tilde{R}}(0,\mu_1) = \text{span}\{\cos(\theta)\},\\
 && Y_1 = \text{Im}\,\mathcal{F}_{\tilde{R}}(0,\mu_1)= \text{span}\{1,\cos(2\theta),\cos(3\theta), \cdots, \cos(n\theta),\cdots\}, \\
 && Y_1 \mbox{$\bigoplus$} \text{Ker}\,\mathcal{F}_{\tilde{R}}(0,\mu_1)  = Y,\\
 && \left[\mathcal{F}_{\mu \tilde{R}}(0,\mu_1)\right] \in \text{span}\{\cos(\theta)\}, \text{ and hence } \left[\mathcal{F}_{\mu \tilde{R}}(0,\mu_1)\right] \cos(\theta) \not\in Y_1.
\eeaa
In other words, all the spaces (kernel space, codimension space, non-tangential space) meet the requirements of the Crandall-Rabinowitz Theorem, and all the assumptions in the Crandall-Rabinowitz Theorem are satisfied. Therefore, $\mu = \mu_1$ is a bifurcation point for the system \re{p1} -- \re{p13}.
\end{proof}

\end{document}